\documentclass[a4paper,12pt]{article}
\usepackage{mathpazo}
\usepackage{epic,eepic,amsfonts,amssymb,amsmath}
\usepackage{epigraph}
\usepackage[official]{eurosym}
\usepackage{mathtools}
\usepackage{xcolor}

\usepackage{pbox}

\newtheorem{proposition}{Proposition}

\newtheorem{lemma}{Lemma}
\newtheorem{corollary}{Corollary}
\newtheorem{theorem}{Theorem}
\newtheorem{definition}{Definition}

\newtheorem{question}{Question}

\newenvironment{proof}[1][Proof]{\noindent\textbf{#1.} }{\ \rule{0.5em}{0.5em}}

\begin{document}
\title{Axioms for Typefree Subjective Probability }
\author{Cezary Cie\'sli\'nski (University of Warsaw)\footnote{c.cieslinski@uw.edu.pl} \\ Leon Horsten (Universit\"at Konstanz)\footnote{Leon.Horsten@uni-konstanz.de} \\ Hannes Leitgeb (Ludwig-Maximilian-Universit\"at M\"{u}nchen)\footnote{leitgebmcmp@icloud.com}}

\date{\today}

\maketitle


\begin{abstract}
\noindent We formulate and explore two basic axiomatic systems of typefree subjective probability. One of them explicates a notion of finitely additive probability. The other explicates a concept of infinitely additive probability. It is argued that the first of these systems is a suitable background theory for formally investigating controversial principles about typefree subjective probability.
\end{abstract}


\section{Introduction}\label{introduction}

Subjective rational probability is intensively investigated in contemporary formal epistemology and confirmation theory. This notion is normally conceived of, either explicitly or implicitly, in a typed way, i.e., as applying to propositions that do not contain the concept of subjective rational probability itself. But formal epistemologists are becoming increasingly interested in typefree (or reflexive) subjective rational probability.

From a logical point of view, the following urgent question then presents itself:
\begin{center}
\textbf{What are basic logical calculi governing \\ typefree subjective rational probability?}
\end{center}

\noindent This is the question that we discuss in this article.

Since we want to have natural ways of constructing self-referential sentences at our disposal (in particular the diagonal lemma), we will formalise subjective probability as a (two-place) \textit{predicate} rather than as a sentential operator, as is the more common practice in confirmation theory and formal epistemology. Our predicate expresses a functional relation between sentences on the one hand, and rational or real numbers on the other hand. The target notion will be the familiar concept of subjective rational probability. This is in contrast with some other recent work on self-refe\-rential probability (such as \cite{CampbellMoore} and \cite{Leitgeb2}) in which a semantic concept of probability is targeted. Moreover, in this article we insist on classical logic governing this subjective probability predicate: first-order classical logic will be relied on throughout.

It should not, perhaps, be assumed that there is a single correct elementary theory of typefree subjective rational probability. Maybe we should instead look for basic calculi that occupy a significant place in a landscape of possible background theories of typefree subjective probability.
Surprisingly, this field is wide open. But this question is important. In order to obtain solid and general results in formal epistemology, rigorous axiomatic frameworks in which controversial epistemological rules and principles are studied, are needed.

We present and discuss two such calculi: one for \textbf{finitely additive} probability, and one for \textbf{$\sigma$-additive} probability. We do not claim that these are the only interesting elementary systems of typefree subjective probability that can be thought of. We investigate some of the proof-theoretic properties of these systems, motivated by an analogy with certain typefree truth theories. We will see that the elementary system for finitely additive probability that we propose can be seen as a \emph{minimal} system of typefree subjective probability, whereas the elementary system for $\sigma$-additive probability that we propose can be seen as a \emph{maximal} system of typefree subjective probability. In a concluding section, we take some first steps in the investigation of controversial epistemic principles against the background of these basic formal calculi. 

In our investigation, we will exploit the analogy between probability and truth: the property of truth is to some extent similar to, albeit of course not identical to, the property of having probability 1. Also, subjective probability can be seen as a quantitative version of the qualitative notion of justified belief. So the theory of reflexive justified belief also contains lessons for the theory of typefree subjective probability.

Our aim is to develop calculi that are in a sense \textbf{elementary}. In particular, we want to keep the languages that we work with as simple as possible.  The only non-logico-mathematical symbol will be one for subjective rational probability ($\mathrm{Pr}$). In this sense, we focus in this article on the \textbf{pure} calculus of typefree subjective probability. Thus we work in a more austere environment than some recent work in this area, in which the relation between truth and probability is investigated in a typefree context (such as \cite{Leitgeb} and \cite{Leitgeb2}). This does not mean that we find these richer frameworks in any way objectionable. But we believe that having a robust sense of what is possible in an austere setting is valuable for research into typefree probability in more expressive settings. Likewise we have of course no objection whatsoever against enriching the language of typefree subjective probability with empirical predicates, although we will not have much to say about that in the sequel.

The technical results in this article must be classed as \emph{basic}. Most of the propositions and theorems are obtained by adapting arguments in the literature for analogous arguments for axiomatic theories of related notions, such as truth, justification, believability. Our aim here is merely to contribute to the \textbf{groundwork} of the theory of typefree subjective probability: much work remains to be done.


\section{Paradox?}\label{paradox}

We will try to exploit, to some extent, the analogy between having subjective rational probability 1, on the one hand, and being true, on the other hand. Since we are interested in typefree probability, the analogy will be with typefree truth.

Typefree truth is a notion that is known to be prone to paradox: intuitive principles (the unrestricted Tarski-biconditionals) lead to a contradiction. What about type-free (subjective) probability?

Two important principles from the literature on axiomatic truth are \textbf{Factivity} and \textbf{Necessitation}. Factivity is the schematic axiom that says that \textit{if A is true, then A}; Necessitation is the following schematic inference rule: \textit{From a proof of A, infer that A is true}.
From the literature on axiomatic truth, we know that Factivity and Necessitation together yield a contradiction. This is known as the Kaplan-Montague paradox:\footnote{See \cite{KaplanMontague}.} it is a mild strengthening of the liar paradox.  The literature on typefree truth theories shows that typefree truth theories divide roughly into two families: Friedman-Sheard-like ($\mathrm{FS}$-like) theories and Kripke-Feferman-like ($\mathrm{KF}$-like) theories. This can be seen as a reflection on whether Factivity or the rule of Necessitation ought to be rejected: $\mathrm{FS}$ rejects Factivity, and $\mathrm{KF}$ rejects Necessitation.\footnote{See for instance \cite[chapter 14, chapter 15]{Halbach2}. In this article we assume familiarity with FS, not with KF.}

For typefree subjective probability, all this means that there is a prima facie reason for being at the same time worried and cautiously hopeful. The basic axioms for subjective rational probability are Kolmogorov's axioms\footnote{See \cite{Kolmogorov}.} for being a finitely additive probability function. One of Kolmogorov's principles says that \emph{necessary} truths should be given probability 1. We want to keep our language as simple as possible, so we do not have a notion of necessity represented in it. Therefore we cannot directly express this principle. But the Necessitation Rule for subjective probability 1, i.e., 
$$   \frac{\vdash \varphi}{\vdash \mathrm{Pr} (\phi) = 1},$$
appears to be a passable approxiomation to (and indeed weakening of) it.\footnote{In the interest of readability, we will be somewhat sloppy with notation, especially regarding coding, in this article.} Since not only the purely mathematical principles about the rational numbers or the real numbers, but presumably also the normative principles that govern subjective rational probability are necessary, this rule should hold for all $\varphi$, including those that include occurrences of $\mathrm{Pr}$.

Thus we have \emph{half} of what is needed to generate a contradiction, i.e., we have reason to be worried. On the other hand, while Factivity seems eminently plausible for truth, it is not clearly a reasonable constraint on probability 1. The only principle concerning subjective probability, considered in the literature, that entails it, is the \textbf{Principle of Regularity}, which says that only necessary truths should be given subjective probability 1. The principle of Regularity is widely rejected as a constraint on rational subjective probability.\footnote{See for instance \cite{Hajek}.} Indeed there is prima facie reason to be suspicious about this principle: for instance, it seems natural to assign probability 1 to propositions that express elementary observational results, which are obviously contingent. In any case, we now already see that the situation is dire for calculi of typefree subjective probability that do include Regularity (such as certain non-Archimedean theories of probability)\footnote{See for instance \cite{BHW}.}, for then, if we accept Necessitation, the Kaplan-Montague argument goes through. 

In the light of these considerations, we reject Factivity and endorse Necessitation. We then cannot fully carry out the Kaplan-Montague argument for probability 1, and can at least hope to avoid contradiction. By rejecting Factivity, we position ourselves in an FS-like, rather than in a KF-like environment. Not only is Factivity to be rejected, but there seems also no reason to trust its weaker cousin Converse Necessitation. Even a \textit{proof} that  a given statement has probability 1 does not give us a compelling reason that that statement is true.

When the truth predicate in FS is interpreted as a concept of probability rather than as truth, the resulting principles are close to a type-free version of the Kolmogorov axioms. Our strategy will therefore be to get as close as consistently possible to the Kolmogorov axioms in a typefree predicate setting, and against a reasonable mathematical background.


\section{Finite and $\sigma$-additive typefree probability}

In this section, we present an elementary formal theory of finite typefree subjective probability, and an elementary formal theory of $\sigma$-additive typefree subjective probability. Moreover, we discuss some elementary properties of these two systems.


\subsection{Languages and background theories}

We will define a basic theory of typefree finitely additive probability and a basic theory of $\sigma$-additive probability. For finitely additive probability we do not need to take limits, so a background theory of the rational numbers suffices. For $\sigma$-additive probability we do need to take limits, so a background theory of the real numbers is needed.

The natural numbers in each case form a significant sub-collection of the domain of discourse. So we assume that each of the two languages contains a predicate $N$ that expresses being a natural number.


\subsubsection{$Q$ and $\mathcal{L}_{Q}$}

Let $Q^-$ be some standard classical theory of the rational numbers, formulated in the language $\mathcal{L}_{Q^-}$, such that it contains the Peano Axioms restricted to $N$. The language $\mathcal{L}_{Q}$ is defined as $\mathcal{L}_{Q^-} \cup \{ \mathrm{Pr} \}$, where $\mathrm{Pr}$ is a two-place predicate such that $\mathrm{Pr}(x,y)$ expresses that the rational subjective probability of $x$ is $y$. We will sometimes write $\mathrm{Pr}(x)=y$ instead.

We assume that, in the finitely additive probability theory that we will define, the logical and nonlogical schemes of $Q^-$ are extended to the language including $\mathrm{Pr}$. This gives rise to the theory $Q$.


\subsubsection{$R$ and $\mathcal{L}_{R}$}

Let $R^-$ be some standard classical theory of the real numbers, formulated in a language $\mathcal{L}_{R^-}$, and let $\mathcal{L}_{R}$ be defined as $\mathcal{L}_{R^-} \cup  \{ \mathrm{Pr} \} $.
Again we assume that, in the probability theories that we will define, the logical and nonlogical schemes of $R^-$ are are extended to the language including $\mathrm{Pr}$. This gives rise to the  theory $R$.


\subsubsection{Coding}

For the language $\mathcal{L}_{Q}$, coding works in the usual way. But there are uncountably many real numbers. To deal with this, we proceed roughly as in \cite{Fujimoto}. 
 Within $R^-$ we can describe the language $\mathcal{L}_{R}^{\infty}$, which contains $\mathcal{L}_{R}$, but also contains constant symbols $c_x$ for each element $x\in \mathbb{R}$. This formalisation of $\mathcal{L}_{R}^{\infty}$ in $\mathcal{L}_{R}$ provides us with a coding of the expressions of $\mathcal{L}_{R}^{\infty}$. For an $\mathcal{L}_{R}^{\infty}$-expression $e$ we denote its code by $\ulcorner e \urcorner$. We specially denote the code of $c_x$ for $x\in \mathbb{R}$ by $\dot{x}$. This formalisation also comes with a coding of various syntactic relations and operations on $x\in \mathbb{R}$.\footnote{An alternative way of proceeding for the language of the real numbers is to work with a probability-satisfaction predicate $\mathrm{Pr}(x,y,z)$, which expresses that the probability of $x$ holding of $y$ is $z$.} As mentioned in the introduction, throughout the article we will often be sloppy in our notation.


\subsection{Finitely additive typefree probability}\label{finite}


We first turn to the principles of the basic theory of \textbf{finitely additive typefree probability}, which we call $\mathrm{RKf}$ (``Reflexive Kolmogorov Finite''). They are expressed in $\mathcal{L}_{Q}$, which is $\mathcal{L}_{Q^-} \cup \{\Pr \}$. Let $\textnormal{Tm}^{\textnormal{c}}$ be the set of constant terms and let $t^{\circ}$ be the value of term $t$ (both notions can be expressed in $\mathcal{L}_{Q}$). The axioms are as follows:

\begin{enumerate}

\item[Kf1$^-$] $Q$

\item[Kf2$^-$] $\mathrm{Pr}$ is a function

\item[Kf3$^-$] $ \Pr(x,y) \rightarrow (x \in \mathcal{L}_{Q}  \wedge 0 \leq x \leq 1) $

\item[Kf4$^-$] $ \forall t \in \textnormal{Tm}^{\textnormal{c}} \big( \Pr(\varphi(t)) = 1 \equiv \varphi(t^{\circ}) \big)$ \hspace{0.5cm} for all $\varphi \in \mathcal{L}_{Q^-}$

\item[Kf5$^-$] $\Pr(x), \Pr(y) \leq \Pr(x \vee y)$


\item[Kf6$^-$]  $ \Pr (x \dot{\vee} y) = \Pr (x) + \Pr (y) - \Pr (x \dot{\wedge} y)    $

\item[Kf7$^-$] $ \Pr (\dot{\neg} x) = 1 - \Pr (x) $

\item[Kf8$^-$] $$  \frac{\vdash \phi}{\vdash \Pr (\phi , 1)} , $$

\end{enumerate}

In these axioms, the free variables are assumed to be universally quantified over. Kf4 is an axiom schema; concrete axioms are obtained from Kf4 by substituting formulas of $\mathcal{L}_{Q^-}$ for the schematic letter $\varphi$.

A comparison with \cite{Kolmogorov} shows that all principles of $\mathrm{RKf}$ except Kf4, Kf5, and Kf8 are Kolmogorov axioms. But Kf4 and Kf8 together aim to approximate the remaining Kolmogorov axiom, viz. the axiom that says that necessary truths have probability 1. In particular, rule Kf8 is justified because a proof of a statement $\phi$ from the (necessary) pure principles of typefree subjective probability entails that that $\phi$ is necessary, and therefore should get probability 1. In Leitgeb's systems of typefree probability, a slightly different necessity principle is adopted, namely $$ Bew_{S} (x) \rightarrow  \Pr (x , 1) , $$ where $S$ is the background system \textit{without} the principles of subjective probability \cite[section 3]{Leitgeb2}. This necessitation principle is of course sound, but it is obviously weaker than Kf8 in specific ways.

Typefree systems can never be fully compositional, since type-freeness precludes an ordinary notion of rank of formulas. Nonetheless, FS has been touted as a highly compositional axiomatic theory of truth.\footnote{See \cite{Halbach}.} The system $\mathrm{RKf}$ is also highly compositional, but slightly less so than FS, for the axiom Kf6 does not explain the truth conditions of probabilities of disjunctions in terms of truth conditions of formulae of lower rank. Axiom Kf5 has been included in order to compensate (to some degree) for this deficiency.

As mentioned in the introduction, $\mathcal{L}_{Q}$ can be extended by empirical vocabulary. Empirical truths can unproblematically be added to $\mathrm{RKf}$ as extra premises. But we do not automatically want to assign all empirical truths probability 1. So we do not simply want to add empirical truths as new \textit{axioms}, for then they would fall in the scope of Kf8.

$\mathrm{RKf}$ is at least minimally capable:

\begin{proposition}

$\mathrm{RKf} \vdash \forall \phi, \psi \big( \Pr (\phi \rightarrow \psi,1) \rightarrow \Pr (\psi) \geq \Pr (\phi) \big)$.

\begin{proof}
Straightforward calculation in $\mathrm{RKf}$.
\end{proof}
\end{proposition}

\noindent It follows in particular that if $\Pr (\phi \rightarrow \psi)=1$ and $\Pr (\phi)=1$, then $\Pr (\psi)=1$. In view of this, we may ask an analogue of a question from provability logic:
\begin{question}
What is the propositional modal logic of probability 1, conceived as a modality?
\end{question}
\noindent From the previous proposition and Rule Kf8 it follows that it is a \textit{normal} propositional modal logic.

Rule Kf8 can be interpreted as saying that $\mathrm{RKf}$ is pointwise self-recom\-men\-ding. Since probability 1 does not entail truth, one might hope that even global self-recommendation does not contradict G\"odel's incompleteness theorems, even though it looks like a (global) reflection principle. However, this is not the case:

\begin{proposition}\label{refl}
 There is no consistent system  $S\supset$ $\mathrm{RKf}$ for which $$  S \vdash \forall x: Bew_S(x) \rightarrow \Pr(x) = 1 .$$

\begin{proof}
Assume $ S \vdash \forall x: Bew_S(x) \rightarrow \Pr(x) = 1 $ and $S$ is consistent. Take, by diagonalisation, a formula $\varphi$ such that $$  S \vdash \varphi \leftrightarrow\neg ( Bew_S(\varphi) \rightarrow \Pr(\varphi) = 1  ).$$ We reason in $S$, and suppose $\varphi$. Then $\neg ( Bew_S(\varphi) \rightarrow \Pr(\varphi) = 1  ),$ which contradicts our assumption. So we have  $S \vdash \neg \varphi .$ By {\normalfont Kf8}, then $S \vdash \Pr ( \neg \varphi ) = 1,$ i.e.,   $S \vdash \Pr (  \varphi ) = 0.$ So $S \vdash \neg Bew_S (\varphi)$, by our assumption. So by the second incompleteness theorem, $S$ is inconsistent.
\end{proof}

\end{proposition}

\noindent  Observe that our proof of the first part of this proposition shows that also the ``local'' version of the principle $\forall x: Bew_S(x) \rightarrow \Pr(x) = 1 ,$ i.e., the scheme $Bew_S(\varphi) \rightarrow \Pr(\varphi) = 1 ,$ is inconsistent.  

For reasonable $S$, the principle $\forall x: Bew_S(x) \rightarrow \Pr(x) = 1 $ should  be \textit{true}, so we should be able consistently add it to $S$. Moreover, this principle looks similar to the uniform reflection principle for $S$. Indeed, we suggest that principles such as these are regarded as a kind of \textit{proof theoretic reflection principles}.

All this suggests the following question, which, as far as we know, is open:
\begin{question}
Let conditional probability be defined in the usual way by the ratio formula.  Is there a consistent system  $S\supset$ $\mathrm{RKf}$ for which $$  S \vdash \forall x: \Pr (Bew_S(x) \neq 0) \rightarrow  \Pr (x \mid Bew_S(x)) =1 ?$$
\end{question}
\noindent Here the antecedent is of course inserted only so as to ensure that the consequent is well-defined.\\

Let us now to turn to the question which principles we can consistently add to $\mathrm{RKf}$.

It is easy to see that adding \textbf{probability iteration principles} to $\mathrm{RKf}$ quickly leads to inconsistency.\footnote{See \cite{Christiano}.} This means that despite its minimality, the principles of $\mathrm{RKf}$ already highly constrain the class of possible extensions. As a simple example, just to see how these arguments go, consider the probabilistic analogue of the $S4$ principle of modal logic, which we call $\mathrm{Pr}4$ :\footnote{Weisberg calls this condition \textit{Luminosity} \cite[p.~184]{Weisberg}.} $$ \Pr (\phi, 1) \rightarrow \Pr (\Pr (\phi, 1),1)    .$$

\begin{proposition}\label{4inconsis}
 $\mathrm{RKf} + \mathrm{Pr}4$ is inconsistent.
 
\begin{proof}
Take a \textbf{probabilistic liar sentence} $\lambda$ such that $\mathrm{RKf} \vdash \lambda \leftrightarrow \neg \Pr (\lambda, 1).$ (Such a $\lambda$ of course exists by the diagonal lemma.) Arguing in $\mathrm{RKf}$, Necessitation of the left-to-right direction yields $ \Pr (\Pr (\lambda, 1) \rightarrow \neg \lambda, 1)$. Distributing $\Pr$ over the conditional gives us $$\Pr (\Pr (\lambda, 1), 1) \rightarrow \Pr ( \neg \lambda, 1).$$ An instance of $\mathrm{Pr}4$ is $ \Pr (\lambda, 1) \rightarrow \Pr (\Pr (\lambda, 1), 1)$. Putting these together gives us $ \Pr (\lambda, 1) \rightarrow \Pr ( \neg \lambda, 1)$, i.e., $ \neg \Pr (\lambda, 1)$. Using the right-to-left direction of the instance of the diagonal lemma, we then have $\lambda$, and by Necessitation $ \Pr (\lambda, 1)$, which contradicts our earlier result.
\end{proof}
\end{proposition}

\noindent Note that proposition \ref{4inconsis} does not entail that there can be no models of RKf that make $ \Pr (\phi, 1) \rightarrow \Pr (\Pr (\phi, 1),1)   $ true for all $\varphi$ (or/and its converse). In our proof, we have applied rule Kf8 to a sentence obtained from $\mathrm{Pr}4$. This is only permitted if  $\mathrm{Pr}4$ is taken as an extra \textit{axiom}.

In a similar way, it can be shown that a form of \textit{negative introspection}, and also its converse, cannot consistently be added to $\mathrm{RKf}$:

\begin{proposition}\label{negintro}

\hspace{0cm}

\begin{enumerate}
 \item The principle $\Pr (x) < 1 \rightarrow \Pr (\Pr (x) < 1) = 1$ cannot consistently be added to $\mathrm{RKf}$;
\item The principle $ \Pr (\Pr (x) < 1) = 1  \rightarrow     \Pr (x) < 1 $ cannot consistently be added to $\mathrm{RKf}$.
\end{enumerate}

\begin{proof}
The simple proofs of 1. and 2. are exactly like the proofs of Theorem 3e and Theorem 3g, respectively, in \cite{SchusterHorsten}.
\end{proof}

\end{proposition}

From the point of view of typefree truth theory, the iteration principles  that are the subject of propositions \ref{4inconsis} and \ref{negintro} are typical analogues of principles that belongs to the $\mathrm{KF}$-family and are incompatible with  $\mathrm{FS}$. It is, as far as we can tell, an open and interesting question what reasonable analogues of $\mathrm{KF}$ for typefree subjective probability would look like.

Propositions \ref{refl}, \ref{4inconsis}, and  \ref{negintro} show that despite the fact that $\mathrm{RKf} $ is a \textit{basic} system of typefree subjective probability, it is nonetheless fairly restrictive. In particular, it is not very tolerant of introspection principles. In this sense, our findings so far are in harmony with the \textit{anti-luminosity} position that Williamson argues for on other grounds \cite[chapter 4]{Williamson}.

It is, however, consistent to add  to $\mathrm{RKf} $ the converse of Pr4 , which we call CPr4, and which is called \textit{Transparency} by Weisberg \cite[p.~190]{Weisberg}:\footnote{Weisberg observes that Luminosity implies Transparency \cite[p.~196, footnote 6]{Weisberg}.}

\begin{proposition}\label{cpr4}

$\mathrm{RKf + CPr4} $ is consistent.
\end{proposition}

\begin{proof}
This follows from the proof of Theorem 5 in \cite{Cieslinski}, see also \cite{Cieslinski1}. The point is that the theory $\mathrm{RKf+CPr4}$ + ``for every $x$, $\Pr(x)$ is either 1 or 0 or 1/2'' is interpretable in the theory of the model $(Q, B_{\omega})$, with $B_{\omega}$ characterized as in Definition 6 of \cite{Cieslinski} (cf. also Definition 13.4.5 of \cite{Cieslinski1}). The interpretation is obtained by translating ``$\Pr(x) = y$'' as ``$\big(y = 1 \wedge B(x) \big) \vee \big(y = 0 \wedge B(\neg x) \big) \vee \big(y = 1/2 \wedge \neg B(x) \wedge \neg B(\neg x) \big)$''. In particular, the truth of the interpretation of Cpr4 follows from the fact that the model $(Q, B_{\omega})$ makes true the reflection axiom (A3) from Definition 4 of \cite{Cieslinski}.
\end{proof}


\subsection{$\sigma$-additive typefree probability}

Basically, our theory of \textbf{$\sigma$-additive typefree probability}, which we will call $\mathrm{RK}\sigma$ (``Reflexive Kolmogorov Sigma'') is like $\mathrm{RKf}$, except that an axiom of $\sigma$-additivity is added. Its principles are:

\begin{enumerate}

\item[K$\sigma$1] $R$

\item[K$\sigma$2] $\mathrm{Pr}$ is a function

\item[K$\sigma$3]  $ \Pr(x,y) \rightarrow (x \in \mathcal{L}_{R}  \wedge 0 \leq x \leq 1) $

\item[K$\sigma$4]  $ \forall t \in \textnormal{Tm}^{\textnormal{c}} \big( \Pr(\varphi(t)) = 1 \equiv \varphi(t^{\circ}) \big)$ \hspace{0.5cm} for all $\varphi \in \mathcal{L}_{R^-}$

\item[K$\sigma$5] $\Pr(x), \Pr(y) \leq \Pr(x \vee y)$


\item[K$\sigma$6] $ \Pr (x \dot{\vee} y) = \Pr (x) + \Pr (y) - \Pr (x \dot{\wedge} y)    $
\item[K$\sigma$7] $$  \frac{\vdash \phi}{\vdash \Pr (\phi , 1)} , $$

\item[K$\sigma$8] $ \Pr(\dot{\exists} x \in N : y(\dot{x}))  = \displaystyle{\lim_{n \to \infty} } \Pr ( y(\underline{0}) \vee \ldots \vee y(\underline{n})   )  $

\end{enumerate}

\noindent In axiom K$\sigma$4, we use the fact that ``internally'' we have names for all real numbers. This will play a role in some of the theorems in the next section.

As before, K$\sigma$6 is a non-compositional axiom.  Axiom K$\sigma$5 is introduced to compensate for this deficiency.


\section{Connection with the Friedman-Sheard system for typefree truth}

We will relate $\mathrm{RKf}$ and $\mathrm{RK}\sigma$ to the Friedman-Sheard theory $\mathrm{FS}$ of typefree truth, which we assume the readers to be familiar with.\footnote{FS was first introduced in \cite{FriedmanSheard}. The locus classicus for the proof-theoretic investigation of $\mathrm{FS}$ is \cite{Halbach}.} But $\mathrm{FS}$ is formulated ``over'' $\mathbb{N}$, whilst $\mathrm{RKf}$ is formulated ``over'' $\mathbb{Q}$, and $\mathrm{RK}\sigma$ is formulated ``over'' $\mathbb{R}$. So when we speak about  $\mathrm{FS}$ from now on, we will assume it to be formulated ``over'' $\mathbb{Q}$ or ``over'' $\mathbb{R}$: the context will make clear which is meant. We will not go into the boring but routine details of how to formulate $\mathrm{FS}$ ``over'' $\mathbb{Q}$ or ``over'' $\mathbb{R}$.

\begin{theorem}\label{theorem2}
$\mathrm{RK}\sigma$ is consistent, $\omega$-inconsistent, but sound for its mathematical sub-language.

\begin{proof}
See \cite{Leitgeb3}.
\end{proof}
\end{theorem}

\noindent The proofs of these properties are more or less ``borrowed'' from the metamathematics of $\mathrm{FS}$. For instance, the proof of the $\omega$-inconsistency of $\mathrm{RK}\sigma$ is a straightforward adaptation of McGee's argument that shows that $\mathrm{FS}$ is $\omega$-inconsistent (see \cite{Halbach}). This should not surprise us. The system $\mathrm{FS}$ is known as ``the most compositional typefree theory of truth''. The systems $\mathrm{RKf}$ and $\mathrm{RK}\sigma$ are also to a high degree compositional, and include the Necessitation rule.\footnote{It is known that given the presence of Necessitation, the Co-Necessitation rule does not make a proof theoretic difference for $\mathrm{FS}$ \cite[p.~322]{Halbach}.}

Theorem \ref{theorem2} shows that $\mathrm{RK}\sigma$ cannot serve as an acceptable background framework for formally investigating debatable principles concerning typefree subjective probability. Despite its mathematical soundness, its $\omega$-inconsistency is, in our opinion, almost as bad as full inconsistency.

The theory $\mathrm{RKf}$ can be trusted: all its theorems can be interpreted as true under the standard interpretation (see Corollary \ref{omegarkf} below). Since, as we have seen in section \ref{finite}, no simple introspection principles (with the exception of CPr4) can be consistently added as axioms to $\mathrm{RKf}$, they do not form a part of the minimal theory of subjective probability. On the other hand,  the theory $\mathrm{RK}\sigma$ is not to be be trusted, as it cannot be interpreted as true under the standard interpretation. In this context, we remind the reader that there is a long history of scepticism towards $\sigma$-additivity as a principle governing subjective probability \cite{deFinetti}, \cite{Howson}.\footnote{In \cite{LeitgebSchurz} it is argued that even for frequency interpretations of probability, $\sigma$-additivity is suspect.}

The connection between $\mathrm{RK}\sigma$ and $\mathrm{FS}$ goes even further than what theorem \ref{theorem2} describes:

\begin{theorem}\label{theorem3}
$\mathrm{FS}$ is relatively interpretable in $\mathrm{RK}\sigma$.

\begin{proof} (Sketch.) 

Let an intermediary system $\mathrm{RKf}\sigma^+$ be defined as $\mathrm{RKf}\sigma$ + $$\forall x,y: \Pr(x,y) \rightarrow (y=0 \vee y=1).$$

Consider the translation $\mu$ which is the homophonic translation for atomic mathematical formulae, commutes with the logical operators but restricts the quantifiers to the natural numbers, and has the following recursive clause for the truth predicate $T$:
$$  \mu (T x) \equiv \Pr ( \mu (x), 1)   .$$
Then $\mu$ is an interpretation of $FS$ in $\mathrm{RKf}\sigma^+$.\footnote{$FS$ is standardly presented as containing not just necessitation but also the conecessitation rule CONEC (from $T(\varphi)$, infer $\varphi$). However, when discussing arithmetical strength, CONEC can be ignored, since it is known that every arithmetical sentence provable in $FS$ can be proved without CONEC. See \cite[section 5]{Halbach}.} In particular, for interpreting the right to left implication in the FS axiom `$\forall \varphi, \psi \big(T(\varphi \vee \psi) \equiv T(\varphi) \vee T(\psi) \big)$' axiom K$\sigma$5 is used.


But $\mathrm{RKf}\sigma^+$ can be interpreted in $\mathrm{RKf}\sigma$ as follows. Let $\theta$ be the translation which  is the homophonic translation for atomic mathematical formulae, commutes with the logical operators, and has the following recursive clause for $\Pr$:
$$  \theta (\Pr (x,y)) \equiv \Pr ( \theta (x) ,y) \wedge (y=0 \vee y=1).$$
Then $\theta$ interprets $\mathrm{RKf}\sigma^+$ into $\mathrm{RKf}\sigma$.

Stringing these two facts together gives us an interpretation of $FS$ in $\mathrm{RKf}\sigma$.
\end{proof}
\end{theorem}

\begin{corollary}
$\mathrm{RK}\sigma$ is at least as strong as the first-order part of Ramified Analysis up to level $\omega$.

\begin{proof}
This follows directly from theorem \ref{theorem3} and the fact that the arithmetical strength of $FS$ is exactly the first-order fragment of Ramified Analysis up to level $\omega$ \cite[section 5]{Halbach}.
\end{proof}
\end{corollary}

\noindent So if $\mathrm{RK}\sigma$ is to be believed---but it isn't!---then just like the notion of set, and the notion of truth, the notion of (typefree) subjective probability has (some) mathematical power.

On the other hand, the notion of probability captured by $\mathrm{RKf}$ does not have mathematical power.

\begin{definition}\label{fsmin}
Let $FS^-$ be like FS but without the quantifier commutation axiom $\forall y : \exists x T y (x) \leftrightarrow T \exists x :  y (x)$. Instead, $FS^-$ contains the axiom schema ``$ \forall t \in \textnormal{Tm}^{\textnormal{c}} \big( \Pr(\varphi(t)) = 1 \equiv \varphi(t^{\circ}) \big)$'' for all formulas $\varphi$ of the base language (without the probability predicate).
\end{definition} 

\noindent The thought is that by moving from FS to $\mathrm{FS}^-$, we remove the mathematical ``sting'' from it, and that moreover $\mathrm{RKf}$  can be interpreted in the conservative system $\mathrm{FS}^-$.

Conservativity of $\mathrm{FS}^-$ can be established by interpreting it the theory RT of iterated truth, which is conservative over its base theory containing PA. Let $L_0$ be the base language; let $L_{n+1}$ be $L_n$ enriched with the new truth predicate $T_n$. (In effect, $L_{n+1}$ contains the truth predicates $T_0, \ldots , T_n$.) A theory $RT_n$ in the language $L_n$ is defined in the following way.

\begin{definition} $RT_0$ is PA. Apart from the axioms of PA, $RT_{n+1}$ contains the following axioms, for every $i \leq n$:
\begin{itemize} 
\item $ \forall t \in \textnormal{Tm}^{\textnormal{c}} \big( T_i(\varphi(t)) \equiv \varphi(t^{\circ}) \big)$ for each $\varphi \in L_0$,
\item $\forall \varphi \in L_i \big(T_i(\neg \varphi) \equiv \neg T_i(\varphi) \big)$,
\item $\forall \varphi, \psi \in L_i \big(T_i(\varphi \vee \psi) \equiv T_i(\varphi) \vee T_i(\psi) \big)$.
\item full induction for formulas of $L_{n+1}$.
\end{itemize}
In addition, $RT_{n+1}$ has the following necessitation rules for every $i \leq n$ and for every $\varphi \in L_i$:

$$  \frac{\vdash \phi}{\vdash T_i (\phi)} $$

\end{definition} 

\begin{lemma}\label{conser} For every $n$, $RT_n$ is conservative over its background mathematical theory.
\end{lemma}

\begin{proof} Let $RT_n \upharpoonright k$ denote the set of theorems of $RT_n$ having proofs with g\"{o}del numbers smaller than $k$. We demonstrate that:

\begin{itemize}
\item[(*)] $\forall k, n \forall M \big( M \models RT_n \rightarrow \exists S (M, S) \models RT_{n+1} \upharpoonright k \big)$.
\end{itemize}

In words: every model of $RT_n$ can be expanded to a model satisfying all the theorems of $RT_{n+1}$ which have proofs with g\"{o}del numbers smaller than $k$. 

For the proof of (*), fix $k$, $n$ and $M \models RT_n$. Define the set $S_n$ as $\{\psi \in L_n: \psi < k \wedge M \models \psi \}$ (so, $S_n$ contains only sentences with g\"{o}del numbers smaller than $k$). Since $S_n$ is the set of true sentences of restricted syntactic complexity, $RT_n$ proves that $S_n$ is consistent. Define $S$ (the intended interpretation of the predicate `$T_n$' of $RT_{n+1}$) as a maximal consistent extension of $S_n$. Note that $S$ is definable in $M$, hence it is fully inductive. 

It is easy to verify that every proof in $RT_{n+1}$ with g\"{o}del number smaller than $k$ contains only sentences true in $(M, S)$. This finishes the proof of (*).

From (*) it follows that each $RT_n$ is conservative over its background mathematical theory.\end{proof} \\

\noindent
\textit{Remark:} The proof of Lemma \ref{conser} employed the idea of expandability of models of $RT_n$ to models of certain fragments of $RT_{n+1}$. It should be emphasized that this does not generalize to full expandability, i.e., it is not true that every model of $RT_n$ is expandable to a model of full $RT_{n+1}$. In fact, the general expandability theorem fails already for $RT_0$ and $RT_{1}$.\footnote{For the proof, see \cite{Cieslinski1}, Theorem 6.0.13, p. 96.}

\begin{lemma}\label{cezary}
$FS^-$ is proof-theoretically conservative over its background mathematical theory $Q$.
\end{lemma}

\noindent This is not surprising, since it is known that the `formalised $\omega$-rule' $$\forall y:\forall x T (y(x)) \rightarrow T( \forall x:(y(x))$$ is the main factor in the mathematical strength of $FS$.\footnote{See \cite{Sheard}.} \\

\begin{proof}
The proof is basically a repetition of the argument for Theorem 14.31 in \cite[p.~181--185]{Halbach2} with only minor changes: we use the functions $g_n$ defined on p.~181 in \cite{Halbach2} to provide an interpretation of fragments of FS without the quantifier axioms (fragments with restricted number of application of Necessitation) in the $RT_n$'s.
\end{proof}

\begin{theorem}\label{theorem 3}
$\mathrm{RKf}$ is proof-theoretically conservative (for the mathematical base language) over its background mathematical theory $Q$.
\end{theorem}

\begin{proof} (Sketch.) 

\noindent Let again $\mathrm{RKf}^+$ be defined as $\mathrm{RKf} + \forall x,y: \Pr(x,y) \rightarrow (y=0 \vee y=1)$, and let $FS^-$ (over $Q$) be just like in definition \ref{fsmin}.  Now consider the translation $\tau$ which is the homophonic translation for atomic mathematical formulae, commutes with the logical operators, and has the following recursive clause for $\Pr$:
$$  \tau (\Pr (x,y)) \equiv (\neg T(\tau (x) ) \wedge y=0) \vee (T(\tau (x) ) \wedge y=1)  .$$ Then $\tau$ is an interpretation of $\mathrm{RKf}^+$ in $FS^-$.  By lemma \ref{cezary}, $FS^-$ is conservative over $Q$.  Stringing these facts together gives us the conservativity result for $\mathrm{RKf}$. 
\end{proof}

\begin{lemma}\label{omegafs} $FS^-$ can be interpreted in the standard model of arithmetic, hence it is $\omega$-consistent.
\end{lemma}

\begin{proof}
The interpretation of $FS^-$ in the standard model of arithmetic is obtained by revision semantics. Let $N$ be the standard model of arithmetic. Define $T_0$ as empty, $T_{k+1} = \{\psi: (N, T_k) \models \psi \}$. Let $T_{\omega}$ be the set of stable sentences, that is:
\begin{center}
$T_{\omega} = \{\psi: \exists m \forall k \geq m (N, T_k) \models \psi \}$.
\end{center}

Define $T$ as a maximal consistent extension of $T_{\omega}$. Then $(N, T) \models FS^-$. Namely, given a proof $(\varphi_0 \ldots \varphi_k)$ in $FS^-$, it can be demonstrated by induction that $\forall i \leq k (\varphi_i \in T_{\omega} \wedge (N, T) \models \varphi_i)$. In particular, in the step for the necessitation rule, we use the fact that $T_{\omega}$ is closed under necessitation.
\end{proof} \\

Since $\mathrm{RKf}$ is interpretable in $FS^-$, we obtain the following corollary.

\begin{corollary}\label{omegarkf} $\mathrm{RKf}$ can be interpreted in the standard model of arithmetic, hence it is $\omega$-consistent.
\end{corollary}

Theorem \ref{theorem3} and Corollary \ref{omegarkf} provide support for the hypothesis that $\mathrm{RKf}$ is an acceptable background framework for formally investigating debatable principles concerning typefree subjective probability, whilst $\mathrm{RK}\sigma$ most definitely is not.  $\mathrm{RKf}$ is a \textit{minimal} system for reflexive subjective probability, whilst $\mathrm{RK}\sigma$  is a \textit{maximal} system for reflexive subjective probability. Both systems represent natural positions in the landscape of systems of reflexive subjective probability.


\section{Probabilistic reflection}

Since $\mathrm{RKf}$ is an acceptable basic theory of typefree subjective probability, it is a suitable formal background against which questions of formal epistemology might be investigated. Let us have look at one example of this.

In \cite{vanFraassen}, van Fraassen proposed and explored the following probabilistic reflection principle:
\begin{definition}[V, ``van Fraassen'']
\begin{equation*}
\begin{split}
[n>0 \wedge \mathrm{Pr}_t(\mathrm{Pr}_{t+n} (\varphi) = a ) \neq 0  ] \rightarrow \\
\mathrm{Pr}_t ( \varphi \mid \mathrm{Pr}_{t+n}(\varphi) = a) = a.
\end{split}
\end{equation*}
\end{definition}
\noindent Here the subscripts of $\Pr$ are real numbers, representing moments in time. Then V imposes a connection between future and current credences. The antecedent is of course needed to ensure that the conditional probability in the consequent is well-defined.

The principle V (and variations on it) has been much discussed in the literature, and enjoys considerable popularity.
Principle V has an air of ill-foundedness. If we think of later credences as determined, perhaps by conditionalisation, by earlier credences, in a way similar to the way in which higher level typed truth predicates are determined by lower level typed truth predicates, then V seems to break type restrictions.

The variant of van  Fraassen's V by setting $n=0$ in V, is truly type-free; and as a coordination principle for probability functions through time, it seems interesting \cite[p.~322]{Christensen}:
\begin{definition}[RV]
\begin{equation*}
\begin{split}
 \mathrm{Pr}(\mathrm{Pr}(\varphi) = a ) \neq 0   \rightarrow 
\mathrm{Pr} ( \varphi \mid \mathrm{Pr}(\varphi) = a) = a.
\end{split}
\end{equation*}
\end{definition}

Nonetheless, RV cannot consistently  be added as a new axiom to $\mathrm{RKf}$:
\begin{proposition}\label{inconsisRV}
$\mathrm{RKf}$ + RV is inconsistent.

\begin{proof}
We reason in $\mathrm{RKf}$ + RV. 

By the diagonal lemma, we may take a sentence $\lambda$ such that $\vdash \lambda \leftrightarrow \Pr (\lambda) <1, $ or, equivalently,  $\vdash \neg \lambda \leftrightarrow \Pr (\lambda) =1. $

Assume, for a reductio, that $\mathrm{Pr}(\mathrm{Pr}(\lambda) = 1 ) \neq  0$. Then, by RV for the case where $a=1$, $\mathrm{Pr} ( \lambda \mid \mathrm{Pr}(\lambda )= 1) = 1,$ which is equivalent to
$\mathrm{Pr} ( \Pr (\lambda) <1 \mid \mathrm{Pr}(\lambda) = 1 ) = 1,$ which is in turn equivalent to $$\frac{\Pr ( \Pr (\lambda) <1 \wedge \mathrm{Pr}(\lambda) = 1)}{\Pr ( {\mathrm{Pr}(\lambda) = 1} )}=1,$$ which yields a contradiction.

So we conclude  $ \vdash \mathrm{Pr}(\mathrm{Pr}(\lambda) = 1 ) =  0$. Then, by the diagonal property, $ \vdash \mathrm{Pr}( \neg \lambda ) =  0, $ which by a Kolmogorov axiom is equivalent to $ \vdash \mathrm{Pr}( \lambda ) =  1. $ By Necessitation, we then get $\vdash \Pr (\mathrm{Pr}( \lambda ) =  1) = 1$, which gives us a contradiction.
\end{proof}
\end{proposition}

\noindent This again illustrates the restrictiveness of even the minimal calculus $\mathrm{RKf}$.

Other variants of van Fraassen's principle V have been considered in the literature. In the light of proposition \ref{inconsisRV}, they should be regarded with suspicion. Indeed, Campbell-Moore considers the following variant V$^*$: 
\begin{center}
$  \Pr_t (\varphi \mid \Pr_\textrm{t+n} (\varphi) \in [a,b]) \in [a,b]  \textrm{ for all } a,b \textrm{ with } a\leq b .$
\end{center}
She shows by a simple diagonal argument:
\begin{proposition}
$RKf + V^*$ is inconsistent.

\begin{proof}
Theorem 1.7.1 in \cite{CampbellMoore1}.
\end{proof}
\end{proposition}

Observe that proposition \ref{inconsisRV} also tells against van Fraassen's principle V. An agent may not update her probability function during some interval $[t, t+n]$, for some $n>0$, because no new evidence has come in to conditionalize on, and because she has in this interval no reasons for adopting a radically different probability function. Then $\Pr_t = \Pr_{t+n}$. But if $\Pr_t$ satisfies RKf, then the pair $\Pr_t, \Pr_{t+n}$ cannot satisfy $\mathrm{RKf}$ + V.

Van Fraassen's principle V has been criticised anyway. Some drug might make one confident that one can fly; if I think I'll take this drug tomorrow, my present conditional confidence that  I'll be able to fly tomorrow, given that tomorrow I'll be quite sure that I can fly, should not be very high \cite[p.~321]{Christensen}. But RV has been taken by many as a law of rational subjective probability. Van Fraassen, for instance, refers to RV as the ``synchronic---I should think, uncontroversial--- part of [V]'' \cite[p.~19]{vanFraassen2}. The point of proposition \ref{inconsisRV} is that the inconsistency of RV can be \textit{proved} from Kolmogorov principles for finitely additive probability in a typefree setting.



\end{document}